\documentclass[12pt]{amsart}

\usepackage{graphicx}
\usepackage{amssymb}
\usepackage{float}
\usepackage{fullpage}

\usepackage{comment}

\newtheorem{theorem}{Theorem}[section]
\newtheorem{lemma}[theorem]{Lemma}
\newtheorem{corollary}[theorem]{Corollary}
\newtheorem{proposition}[theorem]{Proposition}

\theoremstyle{definition}
\newtheorem{definition}[theorem]{Definition}

\theoremstyle{remark}
\newtheorem{remark}{Remark}

\numberwithin{equation}{section}

\newcommand{\R}{{\mathbb R}}

\newcommand{\N}{{\mathbb N}}

\begin{document}

\title[Mertens, dissected]{Mertens' prime product formula, dissected}

\author{Jared Duker Lichtman}
\address{Mathematical Institute, University of Oxford, Oxford, OX2 6GG, UK}

\email{jared.d.lichtman@gmail.com}

\date{March 16, 2021.}


\begin{abstract}
In 1874, Mertens famously proved an asymptotic formula for the product of $p/(p-1)$ over all primes $p$ up to $x$. Observe this product equals the reciprocal sum of all integers composed of prime factors up to $x$.
It is natural to restrict such series to integers with a fixed number $k$ of prime factors. In this article, we obtain formulae for these series for each $k$, which together dissect Mertens' original estimate. The proof is by elementary methods of a combinatorial flavor.
\end{abstract}



\maketitle

\section{Introduction}
We begin with the Euler-Mascheroni constant $\gamma=0.57721\cdots$, defined as the limit of the difference between the harmonic series up to $x$ and $\log x$. The ubiquitous constant $\gamma$ crops up in many contexts, notably, in the $3^\text{rd}$ of three results from a celebrated paper of Mertens \cite{Mert} on the distribution of prime numbers.

As notation, throughout we write $f(x) = O(g(x))$ and $f(x)\ll g(x)$ to mean $|f(x)/g(x)|$ is bounded, while $f(x)\sim g(x)$ means $\lim_{x\to\infty} f(x)/g(x) =1$. Also, let $\log_2 x = \log\log x$, and let $p$ denote a prime number.
\begin{theorem}[Mertens, 1874] \label{thm:Mert3thms}
There exists a constant $\beta>0$ for which
\begin{align}
\sum_{p\le x}\frac{\log p}{p}  = \ \log x \ + \ O(1), \qquad \sum_{p\le x}\frac{1}{p}  & = \ \log_2 x + \beta \ + \ O\Big(\frac{1}{\log x}\Big) \label{eq:mert1p}\\
\prod_{p\le x}\Big(1-\frac{1}{p}\Big)^{-1}  &\sim \ e^\gamma \log x. \label{eq:Mertprod}
\end{align}
\end{theorem}

Here $\beta=0.26149\cdots$ is Mertens' constant, which is known to satisfy
\begin{align}\label{eq:beta}
\beta -\gamma = \sum_p \big(\log(1-\tfrac{1}{p})+\tfrac{1}{p}\big) = - \sum_p\sum_{j\ge2}\frac{p^{-j}}{j} = - \sum_{j\ge2}\frac{Z(j)}{j},
\end{align}
where $Z(s) = \sum_p p^{-s}$ denotes the prime zeta function, for $s>1$. See for instance Theorem 2.7 in \cite[p.50]{MVtext}.

Now by expanding Mertens' prime product in \eqref{eq:Mertprod}, we have
\begin{align}\label{eq:dissect}
\prod_{p\le x}\Big(1-\frac{1}{p}\Big)^{-1} = \prod_{p\le x}\Big(1+\frac{1}{p}+\frac{1}{p^2}+\cdots\Big) = \sum_{P^+(n)\le x}\frac{1}{n}
\end{align}
where $P^+(n)$ denotes the largest prime factor of $n$. 

Consider ``dissecting'' the sum in \eqref{eq:dissect} according to the number of prime factors of $n$ with multiplicity, denoted $\Omega(n)$. Our main result is an asymptotic formula for this dissected sum.

\begin{theorem}\label{thm:main}
For each fixed $k\ge1$, we have
\begin{align}\label{eq:main}
\sum_{\substack{\Omega(n)=k\\P^+(n)\le x}}\frac{1}{n} \ = \ \sum_{j=0}^k\frac{c_{k-j}}{j!}\big(\log_2 x + \beta\big)^j \ & + \ O_k\Big(\frac{(\log_2 x)^{k-1}}{\log x}\Big)
\end{align}
where the sequence $(c_k)_{k=0}^\infty$ is recursively defined by $c_0=1$ and
\begin{align}\label{eq:recursion}
c_k = \frac{1}{k}\sum_{j=2}^k c_{k-j}\,Z(j).
\end{align}
\end{theorem}
Theorem \ref{thm:main} may be viewed as a ``dissection'' of Mertens' prime product formula. Indeed, as shown later in \eqref{eq:dissectterms}, the main term $e^\gamma\log x$ in \eqref{eq:Mertprod} may be expressed as the series over all $k\ge1$ of the the main terms in \eqref{eq:main} (i.e. the sum over $j\le k$).\footnote{Here ``dissection'' is meant to highlight the formal compatibly of main terms. Whereas the estimate \eqref{eq:main} itself does not necessarily hold uniformly over all $k\ge1$.}

We note this terminology was introduced by Pollack \cite{PP}, who dissected a classical mean value theorem of Hall and Tenenbaum. 

\subsection{Uniform estimates via complex analysis} 
Classically, the analogous series to \eqref{eq:main} has been studied, replacing the condition $P^+(n)\le x$ with the more common $n\le x$.

Mertens' $1^{\rm st}$ theorem implies, by induction on each fixed $k\ge1$,
\begin{align}\label{eq:Landau}
\sum_{\substack{\Omega(n) = k \\ n\le x}}\frac{1}{n} \ \sim \ \frac{(\log_2 x)^k}{k!}
\end{align}
as $x\to\infty$, see \cite[p.228]{MVtext}. Note \eqref{eq:Landau} historically attributed to Landau \cite{Landau}. This is another example of dissection, as the sum over all $k$ of each side gives $\sum_{n\le x}\frac{1}{n}$ and $\log x$, respectively. We also note that the asymptotic \eqref{eq:Landau} also holds with $\Omega(n)$ replaced by $\omega(n)$, the number of distinct prime factors of $n$.

However, \eqref{eq:Landau} only holds for fixed $k$. The celebrated theorem of Sathe and Selberg implies the following uniform estimate for $k$ less than $2\log_2 x$.

\begin{theorem}[Sathe--Selberg]
Define $\nu(z) = \frac{1}{\Gamma(z+1)}\prod_p(1-\frac{z}{p})^{-1}(1-\frac{1}{p})^z$, and let $r = k/\log_2 x$. For any $\varepsilon>0$, as $x\to\infty$ we have uniformly for $r \le 2-\varepsilon$,
\begin{align}\label{eq:SatheSelb}
\sum_{\substack{\Omega(n) = k \\ n\le x}}\frac{1}{n} \ \sim \ \nu(r) \frac{(\log_2 x)^{k}}{k!}.
\end{align}
\end{theorem}
To see this, \cite[Theorem 7.19]{MVtext} or \cite[Theorem 6.5]{Tentext} gives an asymptotic in the stated range
\begin{align*}
\sum_{\substack{\Omega(n) = k \\ n\le x}}1 \ = \ \nu(r) \,\frac{x}{\log x}\frac{(\log_2 x)^{k-1}}{(k-1)!}\ \Big(1 + O_\varepsilon\big(\tfrac{k}{\log_2 x}\big)\Big).
\end{align*}
Then \eqref{eq:SatheSelb} follows by partial summation, combined with e.g. the Erd\H{o}s-Sark\"ozy upper bound $O(k^4\,2^{-k}x\log x)$ uniformly for all $x,k\ge1$, see \cite{ErdSark}.
\begin{remark}
As $\nu(r)=1$ only when $r=0,1$, Landau's estimate \eqref{eq:Landau} holds if and only if $k=o(\log_2 x)$ or $k=(1+o(1))\log_2 x$.
\end{remark}
\begin{remark}
\cite[Theorem 6.4]{Tentext} gives an analogous result with $\Omega(n)$ replaced by $\omega(n)$, by substituting the function $\nu(z)$ above with $\lambda(z)=\frac{1}{\Gamma(z+1)} \prod_p(1+\frac{z}{p-1})(1-\frac{1}{p})^z$.
\end{remark}
\begin{remark}
The Sathe--Selberg theorem is proved through contour integration in the complex plane. Recently, Popa \cite{Popa2,Popa3} and Tenenbaum \cite{Tenen} have obtained results by similar analytic methods, for a generalized series that replaces the conditions $\Omega(n)=k$ and $n\le x$ by the condition $p_1\cdots p_k\le x$ over $k$ independent prime variables. Or equivalently, they weight $n$ by the number of its ordered prime factorizations.
\end{remark}

The uniformity coming from sophisticated analytic tools exemplifies the larger tension within mathematics, between proving the strongest results and using the simplest arguments. Of particular interest historically is the case $k=1$, i.e. the Prime Number Theorem. Hadamard and de la Vall\'ee Poussin initially gave proofs in 1896 using complex analysis, and for decades many believed it impossible to prove by elementary means. It came as a great shock when Selberg and Erd\H{o}s did so in 1948. For an intriguing historical account, see Spencer and Graham \cite{SpencGram}.

As such, we emphasize that in Theorem \ref{thm:main}, our particular conditions $\Omega(n)=k, P^+(n)\le x$ in \eqref{eq:main} are directly amenable to elementary methods when $k$ is fixed. Nevertheless, applying analytic tools to \eqref{eq:main} do lend the advantage of uniformity in $k< (2-\varepsilon)\log_2 x$.

\begin{theorem}\label{thm:uniform}
Let $r = k/\log_2 x$. For any $\varepsilon>0$, as $x\to\infty$ we have uniformly for $r \le 2-\varepsilon$,
\begin{align}
\sum_{\substack{\Omega(n)=k\\P^+(n)\le x}}\frac{1}{n} \ \sim \ \nu(r)e^{r\gamma}\, \Gamma(r+1) \ \frac{(\log_2 x)^k}{k!}.
\end{align}
\end{theorem}

Hence by comparision with the Sathe--Selberg theorem, we obtain the following elegant relation between sums over $P^+(n)\le x$ with those over $n\le x$.
\begin{corollary}\label{cor:friablereg}
Let $r = k/\log_2 x$. For any $\varepsilon>0$, as $x\to\infty$ we have uniformly for $r \le 2-\varepsilon$,
\begin{align}
\sum_{\substack{\Omega(n)=k\\P^+(n)\le x}}\frac{1}{n} \ \sim \ e^{r\gamma}\, \Gamma(r+1)\sum_{\substack{\Omega(n)=k\\n\le x}}\frac{1}{n}.
\end{align}
\end{corollary}
\begin{remark}
One may prove an analogous result for $\omega(n)$, with the same factor $e^{r\gamma}\, \Gamma(r+1)$.
\end{remark}

Note the factor $e^{r\gamma}\, \Gamma(r+1)=1$ if and only if $r=0$. Hence Corollary \ref{cor:friablereg} implies
\begin{align}\label{eq:frireg}
\sum_{\substack{\Omega(n)=k\\P^+(n)\le x}}\frac{1}{n} \ \sim \ \sum_{\substack{\Omega(n)=k\\n\le x}}\frac{1}{n}
\end{align}
if and only if $k$ is in the uniform range $k=o(\log_2 x)$. This is an example of friable regularity, in the following sense. Recall an integer $n$ with $P^+(n)\le x$ is called $x$-smooth or $x$-friable.

\begin{definition}
A sequence $(a_{n})_{n\in\N}$ is {\it friably regular}
\footnote{This extends the notion of friable regularity as in \cite{Duff},\cite{FouvTen}, from equality of limits of convergent series to asymptotic equality of (possibly non-convergent) partial sums.} if $\sum_{n\le x}a_{n} \sim \sum_{P^+(n)\le x}a_{n}$ as $x\to \infty$.
\end{definition}
For example, the friable regularity of $(\mu(n)/n)_{n\in\N}$ is equivalent to the prime number theorem. We also extend the definition to families of sequences.
\begin{definition}
A one-parameter family $(a_{n,x})_{n\in\N}$, indexed by $x\in\R$, is {\it friably regular} if $\sum_{n\le x}a_{n,x} \sim \sum_{P^+(n)\le x}a_{n,x}$ as $x\to\infty$.
\end{definition} 
In particular, Corollary \ref{cor:friablereg} implies the family $({\bf 1}_{\Omega(n)=k}/n)_{n\in\N}$, indexed by $k=k(x)$, is friably regular if and only if $k=o(\log_2 x)$.

\subsection{The coefficients $c_k$}

Finally, we emphasize an important feature of the combinatorial approach in Theorem \ref{thm:main}. The recursion in \eqref{eq:recursion} enables rapid computation of the coefficients $c_k$ to high precision, the first few displayed below.

\[\begin{array}{cc|cc}
k & c_k & k & c_k\\
\hline
0,1 & 1,0          & 6 & 0.0108213\cdots\\
2 & 0.226123\cdots & 7 & 0.0054110\cdots\\
3 & 0.058254\cdots & 8 & 0.0027375\cdots\\
4 & 0.044814\cdots & 9 & 0.0013752\cdots\\
5 & 0.020323\cdots & 10 & 0.0006903\cdots
\end{array}\]

At first glance, one might not expect the coefficients $c_k$ arising from \eqref{eq:recursion} to exhibit any particular structure. However, the combinatorial approach shows $c_k$ to satisfy 
exponentially precise asymptotics.
\begin{theorem}\label{thm:ck2k}
The coefficients satisfy $c_k = \eta\,2^{-k} + O(3^{-k})$. Here the constant $\eta$ is given by $\eta = e^{-1}\prod_{p>2}(1-\tfrac{2}{p})^{-1}e^{-2/p} =0.71206\cdots$.
\end{theorem}

\section{Elementary combinatorial proof for $k$ fixed}

In this section we prove Theorem \ref{thm:main}. For $x,s>0$ define the (truncated) zeta functions
\begin{align*}
Z_k(s,x) = \sum_{\substack{\Omega(n)=k\\P^+(n)\le x}}n^{-s}, \qquad Z(s,x) = Z_1(s,x) = \sum_{p\le x}p^{-s}.
\end{align*}

We first express $Z_k(s,x)$ in terms of $Z(s,x)$.
\begin{proposition}\label{prop:Pkpart}
For each $k\ge1$ and any $x,s>0$ we have the identity
\begin{align}\label{eq:Pkpart}
Z_k(s,x) = \sum_{n_1 + 2n_2 +\cdots = k} \prod_{j\ge1} \frac{1}{n_j!}\big(Z(js,x)/j\big)^{n_j}
\end{align}
where the sum ranges over all partitions of $k$.
\end{proposition}
\begin{proof}
For any $x,s>0$ we have a formal power series identity in $z$,
\begin{align*}
\sum_{k\ge0}Z_k(s,x)z^k & = \sum_{P^+(n)\le x}\frac{z^{\Omega(n)}}{n^{s}} = \prod_{p\le x}\Big(1 + \frac{z}{p^s} + \frac{z^2}{p^{2s}}+\cdots\Big) =  \prod_{p\le x}\Big(1 - \frac{z}{p^s}\Big)^{-1}
\end{align*}
since the function $n\mapsto z^{\Omega(n)}/n^s$ is completely multiplicative. Thus expanding Taylor series,
\begin{align}\label{eq:formalpower}
\sum_{k\ge0}Z_k(s,x)z^k & = \exp\Big(-\sum_{p\le x}\log(1-zp^{-s})\Big)  = \exp\Big(\sum_{p\le x}\sum_{j\ge1}\frac{(zp^{-s})^j}{j}\Big) \nonumber\\
& = \exp\Big(\sum_{j\ge1}\frac{Z(js,x)}{j}z^j\Big) = \prod_{j\ge1}\exp\Big(\frac{Z(js,x)}{j}z^j\Big) \nonumber\\
& =\prod_{j\ge1}\sum_{n_j\ge0}\frac{1}{n_j!}\Big(\frac{Z(js,x)}{j}z^j\Big)^{n_j}
= \sum_{k\ge0}z^k \ \sum_{n_1 + 2n_2 +\cdots = k} \prod_{j\ge1} \frac{1}{n_j!}\big(Z(js,x)/j\big)^{n_j}.
\end{align}
Now \eqref{eq:Pkpart} follows by comparing the coefficients of $z^k$.
\end{proof}
\begin{remark}
This proposition generalizes \cite[Proposition 3.1]{JDAlmost}.
\end{remark}

Next, the recursion \eqref{eq:recursion} for $c_k$ leads to the explicit formula,
\begin{align}\label{eq:ckpartdef}
c_k = \sum_{2n_2 +3n_3\cdots = k}\prod_{j\ge2} \frac{1}{n_j!}\big(Z(j)/j\big)^{n_j}
\end{align}
by the following lemma, for the choices $A_1=0$ and $A_j=Z(j)$ when $j\ge2$.

\begin{lemma}\label{lem:general}
Given any sequence $(A_k)_{k=1}^{\infty}$, the sequence $(b_k)_{k=0}^{\infty}$ is given recursively by $b_0=1$ and $b_k=\frac{1}{k}\sum_{j=1}^k b_{k-j}A_j$, if and only if $(b_k)_{k=0}^{\infty}$ is given explicitly as
\begin{align*}
b_k = \sum_{n_1+2n_2+\cdots=k}\prod_{j\ge1}\frac{(A_j/j)^{n_j}}{n_j!}.
\end{align*}
\end{lemma}
Note the (unique) partition of $k=0$ has $n_j=0$ for all $j\ge1$, so $b_0=\prod_{j\ge1}(A_j/j)^0/0! = 1$.
\begin{proof}
We prove the forward direction by induction on $k$ (the reverse direction is similar). 
For $k=1$, we have $b_1 = b_0A_1 = A_1$. Then assuming the claim for each $r<k$,
\begin{align*}
kb_k & =\sum_{r=1}^k b_{k-r}A_r = \sum_{r=1}^k A_r\sum_{n_1+\cdots=k-r}\prod_{j\ge1}\frac{(A_j/j)^{n_j}}{n_j!}\\
&= \sum_{r=1}^k\sum_{n_1+\cdots=k-r} \frac{A_r^{n_r+1}}{r^{n_r}n_r!}\prod_{j\neq r}\frac{(A_j/j)^{n_j}}{n_j!}
= \sum_{r=1}^k\sum_{\substack{n_1+\cdots=k\\n_r\ge1}}rn_r \,\prod_{j\ge1}\frac{(A_j/j)^{n_j}}{n_j!}\\
&= \sum_{n_1+\cdots=k}\prod_{j\ge1}\frac{(A_j/j)^{n_j}}{n_j!}\sum_{\substack{1\le r\le k\\n_r\ge1}}rn_r = k\sum_{n_1+\cdots=k}\prod_{j\ge1}\frac{(A_j/j)^{n_j}}{n_j!}
\end{align*}
In the last step, we dropped the condition $n_r\ge 1$ (since $rn_r=0$ for $n_r=0$) which gives $\sum_{r=1}^k rn_r=k$. Dividing by $k$ completes the induction.
\end{proof}

Now equipped with Proposition \ref{prop:Pkpart} and formula \eqref{eq:ckpartdef} for $c_k$, we now prove Theorem \ref{thm:main}.

\begin{proof}[Proof of Theorem \ref{thm:main}]
For $Z(j,x)$ with $j\ge2$, we trivially bound $\sum_{p> x}p^{-j}$ by $x^{2-j}\sum_{n>x}n^{-2} = O(x^{1-j})$, which gives
\begin{align*}
Z(j,x) := \sum_{p\le x}p^{-j} = Z(j) - \sum_{p> x}p^{-j} \ = \ Z(j) + O(x^{1-j}) \qquad \text{for }j\ge2.
\end{align*}
Thus plugging into the identity for $Z_k(1,x)$, Proposition \ref{prop:Pkpart} with $s=1$ gives
\begin{align*}
Z_k(1,x) \ & = \  \sum_{n_1 + 2n_2 +\cdots = k} \frac{Z(1,x)^{n_1}}{n_1!}\prod_{j\ge2} \frac{1}{n_j!}\Big(\frac{Z(j)+O(x^{1-j})}{j}\Big)^{n_j}
\end{align*}
For any partition of $k$, the binomial theorem implies $\prod_{j\ge2} \frac{1}{n_j!}\Big([Z(j)+O(x^{1-j})]/j\Big)^{n_j}$ equals $\prod_{j\ge2} \frac{1}{n_j!}(Z(j)/j)^{n_j}$ at negligible cost $O_k(1/x)$. Thus
\begin{align}
Z_k(1,x) & = \ \sum_{n_1 + 2n_2 +\cdots = k} \frac{Z(1,x)^{n_1}}{n_1!}\prod_{j\ge2} \frac{1}{n_j!}\big(Z(j)/j\big)^{n_j} \ + \ O_k(\tfrac{Z(1,x)^k}{x}).
\end{align}
Then for $Z(1,x)$, we recall Mertens' $2^\text{nd}$ theorem
\begin{align}
Z(1,x) := \sum_{p\le x}\frac{1}{p} \ = \ \log_2 x + \beta + E(x), \qquad \text{with } E(x) = O\big(\tfrac{1}{\log x}\big)
\end{align}
so plugging in above gives
\begin{align}
Z_k(1,x) \ & = \ \sum_{n_1 + 2n_2 +\cdots = k} \frac{1}{n_1!}\Big(\log_2 x + \beta + E(x)\Big)^{n_1}\prod_{j\ge2} \frac{1}{n_j!}\big(Z(j)/j\big)^{n_j} \ + \ O_k(\tfrac{(\log_2 x)^k}{x})\nonumber\\
& = \ \sum_{n_1=0}^k \frac{1}{n_1!}\big(\log_2 x + \beta\big)^{n_1}\sum_{2n_2 +\cdots = k-n_1}\prod_{j\ge2} \frac{1}{n_j!}\big(Z(j)/j\big)^{n_j} \ + \ O_k\big(E(x)\,(\log_2 x)^{k-1}\big) 
\end{align}
again by the binomial theorem. Here we used $\prod_{j\ge2} \frac{1}{n_j!}(Z(j)/j)^{n_j}  = O_k(1)$. 

Now recalling \eqref{eq:ckpartdef} and $E(x)=  O(1/\log x)$ completes the proof of Theorem \ref{thm:main}.
\end{proof}

From here, we may ``dissect'' Mertens' $3^\text{rd}$ theorem. Indeed by \eqref{eq:dissect},
\begin{align*}
\prod_{p\le x}\big(1-\tfrac{1}{p}\big)^{-1} &= \sum_{P^+(n)\le x}\frac{1}{n} = \sum_{k\ge0} Z_k(1,x)
\end{align*}
and using the asymptotic formula for $Z_k(1,x)$ from Theorem \ref{thm:main},
\begin{align}\label{eq:dissectterms}
\sum_{k\ge0} \sum_{j=0}^k\frac{c_{k-j}}{j!}\big(\log_2 x + \beta\big)^j
& = \sum_{j\ge0}\frac{1}{j!}\big(\log_2 x + \beta\big)^j \sum_{k\ge j}c_{k-j}
 = e^\beta \log x\sum_{m\ge0}c_m \ = \ e^\gamma \log x,
\end{align}
as desired, provided $\sum_{m\ge0}c_m=e^{\gamma-\beta}$. This follows in turn by \eqref{eq:ckpartdef},
\begin{align*}
\sum_{m\ge0}c_m & = \sum_{m\ge0}\sum_{2n_2 +3n_3\cdots = m}\prod_{j\ge2} \frac{1}{n_j!}\big(Z(j)/j\big)^{n_j} = \prod_{j\ge2}\sum_{n_j\ge0}\frac{1}{n_j!}\big(Z(j)/j\big)^{n_j}\\
& = \prod_{j\ge2}\exp\big(Z(j)/j\big) = \exp\Big(\sum_{j\ge2}\frac{Z(j)}{j}\Big) \ = \ e^{\gamma-\beta}
\end{align*}
recalling \eqref{eq:beta}. This shows the claim.

\section{Combinatorial proof of asymptotics for coefficients $c_k$}
In this section we prove a strengthening of Theorem \ref{thm:ck2k}. To this, we first rephrase the recursion \eqref{eq:recursion} for $c_k$.

Let $A_1 = 0$ and $A_k = \sum_p p^{-k}$ for $k\ge2$. Then $c_k$ is recursively defined by $c_0=1$ and
\begin{align}\label{eq:ckrecur}
kc_k = \sum_{j=1}^k c_{k-j}A_j.
\end{align}

Consider the following induced sequences $A_{k,q},c_{k,q}$ for each prime $q$: let $A_{k,2} = A_k$, $c_{k,2} = c_k$; and if $p$ is the prime preceding $q>2$, let
\begin{align}
A_{k,q} &= A_{k,p}-p^{-k} \qquad \qquad \text{for }k\ge1, \label{eq:defPkq}\\
c_{k,q} &= c_{k,p}-p^{-1}c_{k-1,p} \qquad\text{for }k\ge1, \qquad\text{and} \quad c_{0,q}=c_0. \label{eq:defckq}
\end{align}
Explicitly we have
\begin{align}\label{eq:Pkqexplicit}
A_{k,q} = \sum_{r\ge q}r^{-k}\qquad\text{for }k\ge2, \qquad\text{and} \quad
A_{1,q}=-\sum_{p< q}p^{-1}.
\end{align}

\begin{lemma}\label{lem:ckqrecur}
For each prime $q$ and $k\ge0$, we have the recursion
\begin{align}\label{eq:ckqrecur}
kc_{k,q} &= \sum_{j=1}^k c_{k-j,q} A_{j,q}.
\end{align}
\end{lemma}
\begin{proof}
We proceed by induction on the prime $q$. The base case $q=2$ holds by \eqref{eq:ckrecur}.

Now assume \eqref{eq:ckqrecur} for $p<q$. The difference of recursions \eqref{eq:ckqrecur} for $c_{k,p}$ and $p^{-1}\cdot c_{p,k-1}$ is
\begin{align*}
kc_{k,p} - (k-1)p^{-1}c_{k-1,p} & = \sum_{j=1}^{k-1}(c_{k-j,p}-p^{-1}c_{k-j-1,p})A_{j,p} \ + \ c_{0,p}A_{k,p}\\
& = \sum_{j=1}^{k-1}(c_{k-j,p}-p^{-1}c_{k-j-1,p})(p^{-j} + A_{j,q}) \ + \ c_{0,p}(p^{-k} + A_{k,q})\\
& = \sum_{j=1}^{k-1}(p^{-j} c_{k-j,p}-p^{-j-1}c_{k-j-1,p}) + c_{0,p}p^{-k} \ + \ \sum_{j=1}^{k-1}c_{k-j,q}A_{j,q} + c_{0,q}A_{k,q}\\
& = p^{-1}c_{k-1,p} \ + \ \sum_{j=1}^k c_{k-j,q}A_{j,q}
\end{align*}
using \eqref{eq:defPkq}, \eqref{eq:defckq} and telescoping series. Subtracting $p^{-1}c_{k-1,p}$ gives
\begin{align*}
kc_{k,q} = k(c_{k,p} - p^{-1}c_{k-1,p}) = \sum_{j=1}^k c_{k-j,q} A_{j,q}.
\end{align*}
\end{proof}

Note Lemmas \ref{lem:general} and \ref{lem:ckqrecur} together imply
\begin{align}\label{eq:ckqexplicit}
c_{k,q} &= \sum_{n_1+2n_2\cdots=k}\prod_{j\ge1}\frac{(A_{j,q}/j)^{n_j}}{n_j!}
\end{align}
for each prime $q$, $k\ge1$.

Now with the recursion in hand, we bound the induced sequence $c_{k,q}$.

\begin{lemma}\label{lma:ckbnd}
For each prime $q$, we have $c_{k,q} \, \ll_q \, q^{-k}$ as $k\to\infty$.
\end{lemma}
\begin{proof}
Fix $q$ and let $m_k = \max_{j\le k} q^j|c_{j,q}|$. We shall prove $m_k \ll_q 1$, and it suffices to show this along a subsequence, since $m_k$ is itself a non-decreasing sequence. Namely, we consider the indices $k$ for which $m_k = q^k|c_{k,q}|$.


Recalling \eqref{eq:Pkqexplicit}, we have for all $n\ge1$
\begin{align*}
\sum_{1\le j\le n}q^j\,A_{j,q} = qA_{1,q} + (n-1) + \sum_{2\le j\le n}\sum_{r>q}(q/r)^j = n + O_q(1),
\end{align*}
by summing the geometric series, and so the recursion \eqref{eq:ckqrecur} gives
\begin{align*}
kq^k\,|c_{k,q}| = \bigg|\sum_{j=1}^k q^{k-j}\,c_{k-j,q}\cdot q^j\,A_{j,q}\bigg| \ & \le \ m_{k/2}\bigg|\sum_{k/2< j\le k}q^j\,A_{j,q}\bigg| \ + \ m_k\bigg|\sum_{1\le j\le k/2}q^j\,A_{j,q}\bigg|\\
& = \ m_{k/2}\big(k/2+O(1)\big) \ + \ m_k\big(k/2+O(1)\big).
\end{align*}
And by our choice of $k$, we have $m_k = q^k|c_{k,q}|$ and so
\begin{align}
m_k \le  m_{k/2}\big(1+O(1/k)\big).
\end{align}
Hence by induction on $k$, we conclude
\begin{align*}
m_k \ll m_1\prod_{2^i\le k}(1+O(2^{-i})) \ll \exp\sum_{2^i\le k}O(2^{-i}) \ll 1.
\end{align*}
\end{proof}

Since Lemma \ref{lma:ckbnd} holds for every prime $q$ and the sequences $c_{k,q}$ are defined inductively on primes, Lemma \ref{lma:ckbnd} is self-improving. Indeed, for each pair of consecutive primes $p<q$,
\begin{align*}
c_{k,p} - p^{-1}c_{k-1,p} = c_{k,q} = O(q^{-k}).
\end{align*}
In other words, multiplying above by $p^k$ the modified sequence $c_{k,p}':=c_{k,p}p^k$ satisfies $c_{k,p}' - c_{k-1,p}' = O((p/q)^k)$, so $(c'_{k,p})_{k\ge1}$ is a Cauchy sequence for each prime $p$. Hence the limit
\begin{align*}
\eta_p:=\lim_{k\to\infty} c'_{k,p}=\lim_{k\to\infty} c_{k,p}p^k
\end{align*}
exists with $c'_{k,p} = \eta_p +O((p/q)^k)$. That is,
\begin{align}\label{eq:etap}
c_{k,p} = \eta_p\,p^{-k} + O_p(q^{-k}) \qquad\text{for each prime }p.
\end{align}

To summarize, we expanded the definition of $c_{k,q}$ and used a zeroth order expansion for each prime (Lemma \ref{lma:ckbnd}) to prove a first order expansion for every prime simultaneously.

Continuing in this way, we obtain a $h$th order expansion for $c_{k,q}$ by induction on the order $h\ge1$, at each step proving the respective expansion for every prime simultaneously.

\begin{proposition}\label{prop:ckpnexpand}
For any $h\ge1$, we have
\begin{align}\label{eq:ckpn}
c_{k,p_n} \ = \ \sum_{l=0}^{h-1} \eta_{p_{n+l}}^{(n)}\,p_{n+l}^{-k} \ + \ O_{n,h}(p_{n+h}^{-k}),\qquad\text{where}\quad \eta_{p_{n+l}}^{(n)} = \eta_{p_{n+l}}/\prod_{i=0}^{l-1}\big(1- \tfrac{p_{n+l}}{p_{n+i}}\big)
\end{align}
for all $n$ as $k\to\infty$. Here $p_n$ denotes the $n$th prime, and $\eta_p=\lim_{k\to\infty} c_{k,p}p^k$ as in \eqref{eq:etap}.
\end{proposition}
\begin{proof}
We proceed by induction on $h$. The base $h=1$ holds for all $n$ by \eqref{eq:etap}, since $\eta_{p_n}^{(n)}=\eta_{p_n}$.  Now assume \eqref{eq:ckpn} holds with $h$ for every $n$, and write $c_{k,p_n}  =  \sum_{l=0}^{h} \eta_{p_{n+l}}^{(n)}\,p_{n+l}^{-k} \ + \ E_{k,n,h}$. By assumption $E_{k,n,h} \ll p_{n+h}^{-k}$, and we aim to show $E_{k,n,h} \ll p_{n+1+h}^{-k}$. 

By \eqref{eq:defckq} and the induction hypothesis \eqref{eq:ckpn} for $c_{k,p_{n+1}}$,
\begin{align*}
c_{k,p_{n+1}} & = c_{k,p_n}-p_n^{-1}c_{k-1,p_n}\\
\sum_{l=0}^{h-1} \eta_{p_{n+1+l}}^{(n+1)}\,p_{n+1+l}^{-k} \ + \ O(p_{n+1+h}^{-k})
& = \sum_{l=0}^h \eta_{p_{n+l}}^{(n)}(1 - \tfrac{p_{n+l}}{p_n})p_{n+l}^{-k} \ + \ (E_{k,n,h}-p_n^{-1}E_{k-1,n,h})
\end{align*}
Note by definition $\eta_{p_{n+l}}^{(n+1)}=\eta_{p_{n+l}}^{(n)}(1 - \tfrac{p_{n+l}}{p_n})$, and so the above simplifies as
\begin{align*}
O(p_{n+1+h}^{-k}) = E_{k,n,h}-p_n^{-1}E_{k-1,n,h}.
\end{align*}
Thus, similarly as with \eqref{eq:etap}, the modified sequence $E'_{k,n,h}:=E_{k,n,h}p_n^k$ converges as $k\to\infty$ to some limit $\ell_{n,h}$, with $E'_{k,n,h} = \ell_{n,h}+ O((p_n/p_{n+1+h})^{k})$. That is,
\begin{align*}
E_{k,n,h} = \ell_{n,h}\,p_n^{-k}+ O(p_{n+1+h}^{-k})
\end{align*}
On the other hand, $E_{k,n,h} \ll p_{n+h}^{-k}$ forces $\ell_{n,h}=0$. Hence $E_{k,n,h} = O(p_{n+1+h}^{-k})$ as desired.
\end{proof}

Next, we determine the expansion coefficients $\eta_p$ from \eqref{eq:etap}.
\begin{proposition}\label{prop:etap}
For any prime $p$, the coefficient $\eta_p=\lim_{k\to\infty} c_{k,p}p^k$ equals
\begin{align}
\eta_p = e^{-\sum_{q\le p}p/q}\prod_{q>p}(1-\tfrac{p}{q})^{-1}e^{-p/q}.
\end{align}
\end{proposition}
\begin{proof}
Consider the generating function $C_p(z) = \sum_{k\ge0} c_{k,p} z^k$.
On one hand, the explicit formula \eqref{eq:ckqexplicit} for $c_{k,p}$ implies
\begin{align*}
C_p(z) = \sum_{k\ge0} c_{k,p} z^k & = \sum_{k\ge0}z^k \ \sum_{n_1 +2n_2\cdots = k} \prod_{j\ge1} \frac{\big(A_{j,p}/j\big)^{n_j}}{n_j!}
= \prod_{j\ge1} \sum_{n_j\ge0}\frac{1}{n_j!}\big(A_{j,p}z^j/j\big)^{n_j}\\
& = \prod_{j\ge1} \exp\big(A_{j,p}z^j/j\big) = \exp\big(\sum_{j\ge1}A_{j,p}z^j/j\big).
\end{align*}
Then recalling $A_{j,p} = \sum_{q\ge p}q^{-j}$ for $j\ge2$,
\begin{align}
C_p(z) & = \exp\big(zA_{1,p}+\sum_{q\ge p}\sum_{j\ge2}(z/q)^j/j\big) = e^{zA_{1,p}}\exp\big(-\sum_{q\ge p}[\log(1-z/q)+z/q]\big) 
\nonumber\\
& = e^{zA_{1,p}}\prod_{q\ge p}(1-z/q)^{-1}e^{-z/q}. \label{eq:Cexplicit}
\end{align}
On the other hand, by the expansion \eqref{eq:etap} for $c_k$ we have
\begin{align}\label{eq:Ceta}
C_p(z) = \sum_{k\ge0} c_k z^k & = \eta_p\sum_{k\ge0} (z/p)^k \ + \ O_p\Big(\sum_{k\ge0} (z/q)^k\Big)
= \frac{\eta_p}{1-z/p} \ + \ \frac{O_p(1)}{1-z/q}.
\end{align}
since $A_{1,p}=-\sum_{q< p}q^{-1}$. So comparing $C_p(z)$ from \eqref{eq:Cexplicit} and \eqref{eq:Ceta} at the pole $z=p$,
\begin{align*}
\eta_p = \lim_{z\to p} C_p(z)(1-z/p) = e^{pA_{1,p}-1}\prod_{q> p}(1-p/q)^{-1}e^{-p/q}.
\end{align*}
Hence the result follows since $A_{1,p}=-\sum_{q< p}q^{-1}$.
\end{proof}

Finally, we obtain an expansion for the original sequence $c_{k,p_1}=c_k$ to arbitrary order, which gives a considerable refinement of Theorem \ref{thm:ck2k}.
\begin{theorem}
For each prime $q$,
\begin{align*}
c_k = \sum_{p<q}\alpha_p\,p^{-k} \ + \ O_q(q^{-k})
\end{align*}
where $\alpha_p:=e^{-1}\prod_{q\neq p}(1-\tfrac{p}{q})^{-1}e^{-p/q}$. In particular $c_k = \alpha_2\,2^{-k} + O(3^{-k})$.
\end{theorem}
\begin{proof}
Setting $n=1$ in Proposition \ref{prop:ckpnexpand}, the sequence $c_{k,p_1} = c_k$ satisfies
\begin{align*}
c_k = \sum_{p<q}\eta_p^{(1)}\,p^{-k} \ + \ O_q(q^{-k})
\end{align*}
where, by definition of $\eta_p^{(1)}$ in \eqref{eq:ckpn}, Proposition \ref{prop:etap} gives
\begin{align*}
\eta_p^{(1)}:=\eta_p/ \prod_{q<p}(1-\tfrac{p}{q}) = e^{-1}\prod_{q\neq p}(1-\tfrac{p}{q})^{-1}e^{-p/q} = \alpha_p.
\end{align*}
\end{proof}

\section{Analytic proof for $k$ in uniform range}

We prove Theorem \ref{thm:uniform} which quantitative error, which we state below.

\begin{theorem}
Let $r = k/\log_2 x$ and define $\eta(z) = e^{\gamma z}\prod_p (1-\tfrac{1}{p})^z (1-\tfrac{z}{p})^{-1}$. For any $\varepsilon>0$, as $x\to\infty$ we have uniformly for $r \le 2-\varepsilon$,
\begin{align*}
\sum_{\substack{\Omega(n)=k\\P^+(n)\le x}}\frac{1}{n} \ = \ \eta(r)\ \frac{(\log_2 x)^k}{k!}\ \Big(1 + O_{\varepsilon}\big(\tfrac{k}{(\log_2 x)^2}\big)\Big).
\end{align*}
\end{theorem}
\begin{proof}
By Cauchy's residue formula, we have for any $r<2$,
\begin{align}\label{eq:Cauchy1}
Z_k(1,x) = \frac{1}{2\pi i}\int_{|z|=r} f_x(z) \frac{dz}{z^{k+1}},
\end{align}
where $f_x$ is given by the power series
\begin{align*}
f_x(z) & = \sum_{k\ge0}Z_k(1,x)z^k = \sum_{P^+(n)\le x}\frac{z^{\Omega(n)}}{n}\\
&= \prod_{p\le x}\Big(1+\frac{z}{p}+\frac{z^2}{p^2}+\cdots\Big) = \prod_{p\le x}\Big(1-\frac{z}{p}\Big)^{-1}\\
& = (1+O(E(x)))e^{z\gamma} (\log x)^z\prod_{p\le x}\Big(1-\frac{z}{p}\Big)^{-1}\Big(1-\frac{1}{p}\Big)^{z}\\
& = \  (1+O(E(x)))\eta(z) (\log x)^z,
\end{align*}
as $\prod_{p\le x}(1-\frac{1}{p})^{-1} = (1+E(x))e^\gamma \log x$ by Merten's $2^{\rm nd}$ theorem in quantitative form.\footnote{This also follows from the prime number theorem.}

Hence \eqref{eq:Cauchy1} becomes
\begin{align}\label{eq:PkCauchy}
Z_k(1,x) = \frac{1+O(E(x))}{2\pi i}\int_{|z|=r} \eta(z) (\log x)^z\frac{dz}{z^{k+1}}.
\end{align}

The desired main term in Theorem \ref{thm:uniform} is given by evaluating $\eta(z)$ at $z=r$, namely
\begin{align}\label{eq:mainterm}
\frac{\eta(r)}{2\pi i}\int_{|z|=r}(\log x)^z\frac{dz}{z^{k+1}} = \eta(r) \frac{(\log_2 x)^k}{k!}.
\end{align}

For the error we follow the argument in \cite[p.233]{MVtext}, which we provide for completeness. Recall $E(x)\ll 1/\log x$. For $|z|=r = k/\log_2 x$, integration by parts gives
\begin{align}
\frac{1}{2\pi i}\int_{|z|=r}(z-r)(\log x)^z \frac{dz}{z^{k+1}} = \frac{(\log_2 x)^{k-1}}{(k-1)!} - \frac{r(\log_2 x)^k}{k!} \ = \ 0, \label{eq:zero}\\
\text{and}\qquad \eta(z) - \eta(r) - \eta'(r)(z-r) = \int_r^z (z-w)\eta''(w)\,dw \ll |z-r|^2.  \label{eq:bding}
\end{align}
Thus subtracting \eqref{eq:mainterm} from \eqref{eq:PkCauchy}, the error is
\begin{align*}
\eta(r) \frac{(\log_2 x)^k}{k!} \ - \ Z_k(1,x)
\ & \ll \int_{|z|=r}[\eta(r)-\eta(z)](\log x)^z\frac{dz}{z^{k+1}}  \\ 
& \overset{\eqref{eq:zero}}{=} \int_{|z|=r}[\eta(r)-\eta(z)-\eta'(r)(z-r)](\log x)^z\frac{dz}{z^{k+1}} \\
& \overset{\eqref{eq:bding}}{\ll} \int_{|z|=r}|z-r|^2(\log x)^z\frac{dz}{z^{k+1}}\\
&\ll r^{2-k} \int_{-1/2}^{1/2} (\sin \pi\theta)^2 e^{k\cos(2\pi\theta)}\;d\theta \\
&\ll r^{2-k} e^k \int_0^\infty \theta^2 e^{-8k\theta^2}d\theta
\ll r^{2-k} e^k k^{-3/2}\\
& \ = (\log_2 x)^{k-2}(e/k)^k k^{1/2} \ \ll \ k(\log_2 x)^{k-2}/k!
\end{align*}
Here we used $|\sin x|\le x$, $\cos(2\pi\theta) \le 1 - 8\theta^2$ for $|\theta|\le 1/2$, and Stirling's formula.
\end{proof}

\section*{Acknowledgments}
The author is grateful to Paul Kinlaw and James Maynard for stimulating conversations, as well as to Paul Pollack and Carl Pomerance for valuable feedback. In addition, the author thanks the anonymous referee for comments to clarify the paper. The author is supported by a Clarendon Scholarship at the University of Oxford.

\bibliographystyle{amsplain}

\end{document}